\documentclass[11pt,a4paper]{amsart}

\usepackage{tikz}

\usepackage{upgreek}
\usepackage{amsmath, amsthm, verbatim, amsfonts, amssymb, xcolor}
\usepackage{graphics, xspace, enumerate}
\usepackage[a4paper,margin=3cm]{geometry}
\usepackage{marginnote}

\usepackage[bb=boondox]{mathalfa}

\usepackage{graphicx}
\usepackage[a4paper,colorlinks=true,citecolor=green,urlcolor=green,linkcolor=green,bookmarksopen=true,unicode=true,pdffitwindow=true]{hyperref}

\hypersetup{pdfauthor={}}
\hypersetup{pdftitle={Stability of the overdamped Langevin equation in double-well potential}}

\theoremstyle{plain}
\newtheorem{theorem}{Theorem}[section]
\newtheorem{corollary}[theorem]{Corollary}
\newtheorem{lemma}[theorem]{Lemma}
\newtheorem{proposition}[theorem]{Proposition}
\theoremstyle{definition}
\newtheorem{remark}[theorem]{Remark}

\newcommand{\bt}{\begin{theorem}}
\newcommand{\et}{\end{theorem}}	
\newcommand{\bp}{\begin{proof}}
	\newcommand{\ep}{\end{proof}}
\newcommand{\D}{\mathrm{d}}
\newcommand{\e}{\mathrm{e}}
\newcommand{\V}{\mathcal{V}}
\newcommand {\I} {\ensuremath{\mathbb{I}}}
\newcommand {\R} {\ensuremath{\mathbb{R}}}
\newcommand {\Q} {\ensuremath{\mathbb{Q}}}
\newcommand {\ZZ} {\ensuremath{\mathbb{Z}}}
\newcommand {\N} {\ensuremath{\mathbb{N}}}

\newcommand{\process}[1]{\{#1_t\}_{t\geq0}}

\newcommand{\be}{\begin{equation}}
\newcommand{\ee}{\end{equation}}
\newcommand{\ben}{\begin{equation*}}
\newcommand{\een}{\end{equation*}}
\makeatletter
\newcommand{\ba}{\@ifstar{\@bas}{\@ba}}
\def\@ba#1\ea{\begin{align}#1\end{align}}
\makeatletter

\makeatletter
\newcommand{\ban}{\@ifstar{\@bans}{\@ban}}
\def\@ban#1\ean{\begin{align*}#1\end{align*}}
\makeatletter

\numberwithin{equation}{section}

\newcommand{\Prob}{\mathbb{P}}
\newcommand{\Exp}{{\mathbb{E}}}

\begin{document}
\allowdisplaybreaks[4]

\title{Stability of the overdamped Langevin equation in double-well potential}

\author[Nikola\ Sandri\'{c}]{Nikola Sandri\'{c}}
\address[Nikola\ Sandri\'{c}]{Department of Mathematics\\University of Zagreb\\10000 Zagreb\\Croatia}
\email{nsandric@math.hr}


\begin{abstract} In this article, we discuss stability of the one-dimensional overdamped Lange\-vin equation in double-well potential. We determine unstable and stable equilibria, and discuss the rate of convergence to stable ones. Also, we derive conditions for stability of general diffusion processes which generalize the classical and well-known results of Khasminskii (\cite{Khasminskii-Book-2012}).
   
\end{abstract}

\maketitle

%
%
%
%


\section{Introduction}

The \textit{Langevin equation} is a stochastic differential equation  describing the dynamics of a particle immersed  in a fluid, subjected to an external potential force field and   collisions with the molecules of the fluid:
\be
\begin{aligned}\label{eq:langevin_undamped} 	
	m\,\D X_t&=P_t\D t\\
 \D P_t&=-(\uplambda/m) P_t\D t-\nabla V(X_t)\D t+\upsigma(X_t)\D B_t,\qquad (X_0,P_0)\in\R^d\times\R^d.
\end{aligned} 
\ee Here, $\process{X}$ and $\process{P}$ denote, respectively,   the position and momentum of the particle, $m$ is particle's mass,  $-(\uplambda/m) P_t\D t$, $\uplambda>0$, is the velocity-proportional damping (friction) force, $V$ is particle's potential and  $\upsigma(X_t)\D B_t$ is the  noise term
representing the effect of the collisions with the molecules of the fluid, where $\process{B}$ denotes a standard $d$-dimensional Brownian motion. Observe that here we assume  the measure of the noise strength $\upsigma$ is non-constant,  meaning that the effect of collisions depends on the position of the particle (e.g. due to heterogeneity of the fluid). In this case, the function $\upsigma$ models the nature of the position-dependence.

In the case when the inertia of the particle is negligible in comparison with the damping  force (due to friction), the trajectory 
 of the particle is described by the so-called \textit{overdamped Langevin equation}:
\be\label{eq:langevin_overdamped}\uplambda\,\D X_t=-\nabla V(X_t)\D t+\upsigma(X_t)\D B_t,\qquad X_0\in\R^d.\ee
 Namely, in \cite[Chapter 10]{Nelson-Book-1967}  it has been shown that (under certain assumptions on the potential $V$ and diffusion coefficient $\upsigma$)   the solution to \eqref{eq:langevin_undamped} converges a.s. to the solution to \eqref{eq:langevin_overdamped}, as   $m\searrow0$.

 The main purpose of this article is to discuss stability of the solution to the one-dimensional overdamped Langevin equation in \textit{double-well} or \textit{Landau potential} $V(x)=-ax^2/2+bx^4/4$, $a,b>0$: \be\label{eq:langevin}\uplambda\, \D X_t=(-bX_t^3+aX_t)\D t+\upsigma(X_t)\D B_t,\qquad X_0\in\R.\ee

This potential is  of considerable interest in quantum mechanics and quantum field theory for the exploration of various physical phenomena or mathematical properties since it permits in many cases explicit calculation without over-simplification (see e.g. \cite{Coleman-Book-1979} and  \cite{Liang-Muller-Kirsten-1992}). Typical example where it occurs is in the so-called \textit{ammonia inversion phenomenon}. This is a switching of the nitrogen atom from above to below the hydrogen plane. More precisely, the ammonia molecule is pyramidal shaped with the three hydrogen atoms forming the base and the nitrogen atom at the top. The nitrogen atom sees a double-well potential with one well on either side of the hydrogen plane. Because the potential barrier is finite, it is possible for the nitrogen atom to tunnel through the plane of the hydrogen atoms, thus ``inverting'' the molecule (see \cite{Lehn-Book-1970} for more details).

 For the sake of simplicity, but without loss of generality, in the sequel we assume $a=b=\uplambda=1.$ Also, we impose the following assumptions on the diffusion coefficient $\upsigma$:
 
 	\medskip
 
 \begin{description}
 	\item[A1] $\upsigma$ is locally Lipschitz continuous;
 	
 	\medskip
 	
 	\item[A2] $\displaystyle\limsup_{|x|\nearrow\infty}|\upsigma(x)|/|x|^2<\sqrt{2}.$
 \end{description} 
 
 	\medskip
 	
 Under (\textbf{A1}) and (\textbf{A2}), in  
 \cite[Theorem 3.1 and  Proposition 4.2]{Albeverio-Brzezniak-Wu-2010} and \cite[Theorem 3.1.1]{Prevot-Rockner-Book-2007}
  it has been shown that the equation in \eqref{eq:langevin} admits a unique non-explosive strong solution $\process{X}$ which, in addition, is a temporally homogeneous strong Markov process with continuous sample paths.  Furthermore, in \cite[Remark 2.2 and Proposition 4.3]{Albeverio-Brzezniak-Wu-2010} it has been also shown that  $\process{X}$ is a $\mathcal{C}_b$-Feller process and that for any $f\in \mathcal{C}^2(\R)$ the process \be\label{eq:martingale}M_t^f:=f(X_t)-f(X_0)-\int_{0}^t\mathcal{L}f(X_s)\D s,\qquad t\ge0,\ee is a local martingale, where \be\label{eq:operator}\mathcal{L}f(x)=(-x^3+x)f'(x)+\frac{\upsigma^2(x)}{2}f''(x),\qquad f\in \mathcal{C}^2(\R).\ee Recall, \textit{$\mathcal{C}_b$-Feller property} means that the semigroup $\process{P}$ of $\process{X}$, defined as  \ben P_tf(x):=\int_{\R}f(y)p^t(x,\D y),\qquad t\ge0,\ x\in\R,\ f\in \mathcal{B}_b(\R),\een maps $\mathcal{C}_b(\R):=\mathcal{C}(\R)\cap \mathcal{B}_b(\R)$ to $\mathcal{C}_b(\R)$. Here, $p^t(x,\D y)$ and $\mathcal{B}_b(\R)$ denote, respectively, the transition kernel of  $\process{X}$ and the space of bounded Borel measurable functions.

 \subsection{Stability of the deterministic overdamped Langevin equation \eqref{eq:langevin}}

 We consider \be\label{eq:deterministic}\dot{x}=-x^3+x,\qquad x(0)\in\R.\ee It is easy to check that \eqref{eq:deterministic} admits three solutions: $x_1(t)\equiv0$ (corresponding to the initial condition  $x_1(0)=0$), \ben x_2(t)=-\frac{\e^{t}}{\sqrt{\e^{2t}-1+1/x_2(0)^2}},\qquad x_2(0)<0,\een and  
 \ben x_3(t)=\frac{\e^{t}}{\sqrt{\e^{2t}-1+1/x_3(0)^2}},\qquad x_3(0)>0.\een 
 
 Now, recall that $x_e\in\R$ is called \textit{equilibrium state} to the Cauchy problem \be\label{eq:cauchy}\dot{x}=f(x),\qquad x(0)\in\R,\ee if $x(0)=x_e$ implies that $x(t)\equiv x_e$ (or, equivalently, if $f(x_e)=0$). Clearly, the only equilibria to \eqref{eq:deterministic} are $-1$, $0$ and $1$.  Further, an equilibrium $x_e$ to \eqref{eq:cauchy} is called \textit{stable} if for every $\upvarepsilon>0$ there is $\updelta>0$ such that $|x(0)-x_e|<\updelta$ implies  $|x(t)-x_e|<\upvarepsilon$ for all $t\ge0$; otherwise it is called \textit{unstable}. An equilibrium $x_e$ to \eqref{eq:cauchy} is called \textit{asymptotically stable} if it is stable, and if there is $\updelta>0$ such that whenever $|x(0)-x_e|<\updelta$ then \ben\lim_{t\nearrow\infty}|x(t)-x_e|=0,\een and it is called \textit{exponentially stable} if it is  stable, and if there is $\updelta>0$ such that whenever $|x(0)-x_e|<\updelta$ then \ben\lim_{t\nearrow\infty}\e^{\upkappa t}|x(t)-x_e|=0\een for some $\upkappa>0.$
 Clearly, $-1$ and $1$ are exponentially stable for any $0<\kappa<2$, and $0$ is unstable equilibrium to \eqref{eq:deterministic}.  
 
 Previous discussion suggests that in the non-deterministic setting the states $-1$, $0$ and $1$ might also play an important role. However, in this setting, due to the random term $\upsigma(X_t)\D B_t$ which can ``regularize'' the equation, these points will not necessarily be equilibria of  $\process{X}$:  if $\upsigma$ is ``regular'' enough, i.e. if $\upsigma$ does not vanish at $-1$, $0$ and $1$, $\process{X}$ will admit only one equilibrium  which does not explicitly depend on $-1$, $0$ and $1$.

  \subsection{Stability of the overdamped Langevin equation \eqref{eq:langevin}}
 In the non-deterministic setting the role of equilibria take invariant measures of the underlying process.  A probability measure $\uppi$ on $\R$ is \textit{invariant} for $\process{X}$ if \ben \int_{\R}p^t(x,\D y)\uppi(\D x)=\uppi(\D y),\qquad t\ge0.\een In other words, under $\uppi$ as an initial distribution the marginals of $\process{X}$ do not change over time, i.e. $\process{X}$ is a stationary process.
 
 As the first main result of this article we show that  $\process{X}$ admits at least one equilibrium (invariant measure). 
 
  \begin{theorem}\label{tmy} 
  	Assume (\textbf{A1}) and (\textbf{A2}).	Then  $\process{X}$ admits an invariant measure.
  \end{theorem}

 \noindent Furthermore, we also show that if $\upsigma$ is ``regular'' enough, then  $\process{X}$ admits a unique equilibrium.
 
 \begin{theorem}\label{tmx} 
 Assume (\textbf{A1}) and (\textbf{A2}).	If there is an open interval $I$ containing $-1,$ $0,$ and $1$, such that $\inf_{x\in I}\upsigma(x)>0$, then  $\process{X}$ admits a unique equilibrium $\uppi$ such  that for any $\upkappa>0$, \ben\lim_{t\nearrow\infty}\e^{\upkappa t}\lVert p^t(x,\D y)-\uppi(\D y)\rVert_{\mathcal{TV}}=0,\qquad x\in\R,\een where  $\lVert\cdot\lVert_{\mathcal{TV}}$ stands for the total variation norm on the space of signed measures.
 \end{theorem}

 On the other hand, if $\upsigma$ vanishes at $x_e\in\{-1,0,1\}$, then, obviously,  $X_t=x_e$, $t\ge0$, is a solution to \eqref{eq:langevin}, i.e. $\updelta_{x_e}$ is an invariant measure for $\process{X}$ . The point $x_e$ is said to be \textit{stable in probability} if for any $\upvarepsilon>0$, \ben\lim_{x\to x_e}\Prob^x\left(\sup_{t>0}|X_t-x_e|>\upvarepsilon\right)=0;\een otherwise it is called \textit{unstable}.
 It is called \textit{asymptotically stable in probability} if it is stable in probability and \ben\lim_{x\to x_e}\Prob^x\left(\lim_{t\nearrow\infty}|X_t-x_e|=0\right)=1.\een
 We then conclude the following.

\bt\label{tm1.1} Assume (\textbf{A1}), (\textbf{A2}) and that $\upsigma$ has a root at $x_e\in\{-1,0,1\}$.

\medskip

	\begin{itemize}
		\item[(i)] If $x_e=0$ and \ben\inf_{\upvarepsilon>0}\inf\left\{\upkappa:\frac{|\upsigma(x)|}{|x|}\le\upkappa,\ 0<|x|<\upvarepsilon\right\}<\sqrt{2},\een then $x_e$ is unstable.  Moreover, there is $\upvarepsilon>0$ such that $\Prob^x (\sup_{t\ge0}|X_t|<\upvarepsilon)=0$ for every $0<|x|<\upvarepsilon.$
		
		\medskip
		
		\item[(ii)] If $x_e=0$ and there is $\updelta>0$ such that $|\upsigma(x)|=\sqrt{2}|x|$ for $|x|<\updelta,$ then $x_e$ is unstable. Also, there is $0<\upvarepsilon<\updelta$ such that $\Prob^x (\sup_{t\ge0}|X_t|<\upvarepsilon)=0$ for every $0<|x|<\upvarepsilon.$ 
		
		\medskip
		
		\item[(iii)] If $x_e=0$ and \ben\inf_{\upvarepsilon>0}\sup\left\{\upkappa:\frac{|\upsigma(x)|}{|x|}\ge\upkappa,\ 0<|x|<\upvarepsilon\right\}>\sqrt{2},\een then $x_e$ is asymptotically stable in probability.
		
		\medskip
		
		\item [(iv)] If $x_e\in\{-1,1\}$, then $x_e$ is asymptotically stable in probability. Furthermore, if \ben c:=\upalpha\inf_{x\in\R,\,xx_e>1}\left(x(x+x_e)-(\upalpha-1)\frac{\upsigma(x)^2}{2}|x-x_e|^{-2}\right)>0\een for some $\upalpha>0$ (which is always the case for $0<\upalpha\le1$), then $\Exp^x(|X_t-x_e|^\upalpha)\le |x-x_e|^\upalpha \e^{-c t}$ for $x\in\R$, $xx_e\ge1$, and $t\ge0$, and \ben\lim_{t\nearrow\infty}\frac{\ln |X_t-x_e|}{t}\le -\frac{c}{\upalpha}\qquad \Prob^x\text{-a.s.}\een for all $x\in\R$, $xx_e\ge 1.$ 
	\end{itemize}
		
\et

\subsection{Stability of general diffusion processes}
At the end we discuss stability of general multidimensional  diffusion processes.

\begin{theorem}\label{tm3.1}
	Assume that  equation \be\label{eq:3.1} \D X_t=b(X_t)\D t+\upsigma(X_t)\D B_t,\qquad X_0\in\R^d,\ee admits a unique, non-explosive strong solution which is a strong Markov process with continuous sample paths, where  $b:\R^d\to\R^d$ and $\upsigma:\R^{d}\to\R^{d\times n}$ are  continuous, and $\process{B}$ stands for a standard $n$-dimensional Brownian motion.  Further, assume that there is $x_e\in\R^d$  such that $b$ and $\upsigma$ vanish at $x_e$ (hence, $X_t=x_e$, $t\ge0$, is a solution to \eqref{eq:3.1}).  If there are  $c>0$, and   concave, continuously differentiable and strictly increasing around the origin   function $\upvarphi:[0,\infty)\to[0,\infty)$, with $\upvarphi(0)=0$, such that 
	
	\medskip
	
	\begin{itemize}
		\item [(i)] the function $\Phi_c(t):=c^{-1}\int_1^t\D s/\upvarphi(s)$ maps $(0,\infty)$ onto $\R$;
		
		\medskip
		
		\item[(ii)] $\mathcal{L}\V(x)\le -c\,\upvarphi\circ\V(x)$ for $|x-x_e|>0$, where \ben\mathcal{L}f(x)=\langle b(x),\nabla f(x)\rangle+\frac{1}{2}{\rm Tr}\,\upsigma(x)\upsigma'(x)\mathcal{D}^2f(x),\qquad f\in\mathcal{C}^2(\R^d)\een and $\V(x):=|x-x_e|^\upalpha$ for some $\upalpha>0$,
	\end{itemize}
	
	\medskip
	
\noindent	then \ben |X_t-x_e|^\upalpha\le \Phi_c^{-1}(\Phi_c(Y_x)-t),\qquad\Prob^x\text{-a.s.},\ x\in\R^d,\ t\ge0,\een where $Y_x$ is a strictly positive $\Prob^x$-finite random variable. 
\end{theorem}

	As a consequence of Theorem \ref{tm3.1} we get a generalization of \cite[Theorem 5.15]{Khasminskii-Book-2012} (where it is assumed that $\mathcal{L}\V(x)\le-c\V(x)$ for some $c>0$ and all $|x-x_e|>0$). Also, we conclude super-geometric stability result.
	
	\begin{corollary}\label{c1.1} Assume the conditions of Theorem \ref{tm3.1}. 
		\begin{itemize}
			\item [(i)] If 	\ben\mathcal{L}\V(x)\le \left\{\begin{array}{cc}
				-c\V(x), &  0<\V(x)\le r \\
				-cr,&
				\V(x)\ge r,
			\end{array}\right. \een for some $c>0$ and $r>0,$ then \ben \limsup_{t\nearrow\infty}\frac{\ln|X_t-x_e|}{t}\le-\frac{c}{\upalpha}.\een
			\item[(ii)] If there are $c>0$, $\upbeta>1$ and $0<r_\upbeta\le\e^{1/\upbeta-1}$ such that $\mathcal{L}\V(x)\le-c\upvarphi_{\upbeta}\circ\V(x)$ for  all $|x-x_e|>0$, where 
			\ben\upvarphi_{\upbeta}(t):=\left\{\begin{array}{cc}
				\upbeta t\left(-\ln t\right)^{1-1/\upbeta}, & 0\le t\le r_\upbeta \\
				\upbeta r_\upbeta(-\ln r_\upbeta)^{1-1/\upbeta},&
				t\ge r_\upbeta,
			\end{array}\right. \een
			then \ben \limsup_{t\nearrow\infty}\frac{\ln|X_t-x_e|}{t^\upbeta}\le-\frac{c^\upbeta}{\upalpha}\een
			($r_{\upbeta}$ is chosen  such that $\upvarphi_{\upbeta}$ is non-decreasing).
		\end{itemize}
	\end{corollary}

The remainder of the article is organized as follows. In Section \ref{s2}, we discuss existence and uniqueness of invariant measures of $\process{X}$, and prove Theorems \ref{tmy}, \ref{tmx} and  \ref {tm1.1}. In Section \ref{s3}, we prove Theorem \ref{tm3.1} and Corollary \ref{c1.1}, and discuss super-geometric stability of general diffusion processes.

 \section{Stability of the overdamped Langevin equation \eqref{eq:langevin}}\label{s2}

In this section, we discuss existence and uniqueness of invariant measures of $\process{X}$.
  For $a\in\R$ and $r>0$ denote by $I_r(a)$ the open $r$-interval around $a$, i.e. $I_r(a):=(a-r,a+r).$  Also, $\bar{I}_{r}(a)$ and $I^c_{r}(a)$ denote, respectively, the closure and complement of $I_{r}(a)$.

\begin{proof}[Proof of Theorem \ref{tmy}.] 
	According to \cite[Theorem 3.1]{Meyn-Tweedie-AdvAP-II-1993}
it suffices to prove that for each $x\in\R$ and $0<\upvarepsilon<1$, there is a compact set $K\subset\R$ such that \ben\liminf_{t\nearrow\infty}\frac{1}{t}\int_0^tp^s(x,K)\D s\ge 1-\upvarepsilon.\een	In order to show this, define $\mathcal{V}(x):=x^2$ and observe that \ben \mathcal{L}\V(x)=-2x^4+2x^2+\upsigma(x)^2.\een Now, according to (\textbf{A2}), there are  $0<\updelta<2$ and $r_\updelta>2/\sqrt{\updelta}$, such that $\upsigma(x)^2\le (2-\updelta)|x|^4$ for $|x|\ge r_\updelta.$
	Thus,
	 \ban \mathcal{L}\V(x)=&-2x^4\mathbb{I}_{\bar{I}_{r_\updelta}(0)}(x)-2x^4\mathbb{I}_{\bar{I}^c_{r_\updelta}(0)}(x)+2x^2\mathbb{I}_{\bar{I}_{r_\updelta}(0)}(x)+2x^2\mathbb{I}_{\bar{I}^c_{r_\updelta}(0)}(x)\\&+\upsigma(x)^2\mathbb{I}_{\bar{I}_{r_\updelta}(0)}(x)+\upsigma(x)^2\mathbb{I}_{\bar{I}^c_{r_\updelta}(0)}(x)\\
	 \le&-\frac{\updelta}{2}x^4\mathbb{I}_{\bar{I}^c_{r_\updelta}(0)}(x)+(-2x^4+2x^2+\upsigma(x)^2)\mathbb{I}_{\bar{I}_{r_\updelta}(0)}(x)\\
	 \le& -\frac{\updelta}{2}r_\updelta^4+\left(\frac{\updelta}{2}r_\updelta^4+2r_\updelta^2+\upsigma(x)^2\right)\mathbb{I}_{\bar{I}_{r_\updelta}(0)}(x).
	\ean
By	denoting $\upsigma:=\sup_{x\in\bar{I}_{r_\updelta}(0)}|\upsigma(x)|$, 
we get $$\mathcal{L}\V(x)\le -\frac{\updelta}{2}r_\updelta^4+\left(\frac{\updelta}{2}r_\updelta^4+2r_\updelta^2+\upsigma^2\right)\mathbb{I}_{\bar{I}_{r_\updelta}(0)}(x).$$ Analogously, for $r>r_\updelta$ we conclude $$\mathcal{L}\V(x)\le -\frac{\updelta}{2}r^4+\left(\frac{\updelta}{2}r^4+2r^2+\upsigma^2+(2-\updelta)r^4\right)\mathbb{I}_{\bar{I}_{r}(0)}(x),$$
where we used the fact \ben|\upsigma(x)|^2\mathbb{I}_{\bar{I}_{r}(0)}\le(\upsigma^2+(2-\updelta)r^4)\mathbb{I}_{\bar{I}_{r}(0)}.\een Now, from \cite[Theorem 1.1]{Meyn-Tweedie-AdvAP-III-1993} we conclude that for each $x\in\R$ and $r>r_\updelta$ we have \ben\liminf_{t\nearrow\infty}\frac{1}{t}\int_0^tp^s(x,\bar{I}_{r}(0))\D t\ge\frac{(\updelta/2)r^4}{(\updelta/2)r^4+2r^2+\upsigma^2+(2-\updelta)r^4}.\een
The assertion now follows by choosing $\updelta$ close to $2$ and $r$ large enough.
	\end{proof}

Standard assumptions which ensure uniqueness  of an invariant measure are strong Feller property and open-set irreducibility. Recall, $\process{X}$ is called 

\medskip

\begin{itemize}
	\item [(i)] \textit{strong Feller} if $P_tf\in \mathcal{C}_b(\R)$ for any $t>0$ and $f\in\mathcal{B}_b(\R).$
	
	\medskip
	
	\item[(ii)] \textit{open-set irreducible} if for any $x\in\R$ and open set $O\subseteq\R$, \ben\int_{0}^\infty p^t(x,O)\D t>0.\een\end{itemize} 

\medskip
	
\noindent According to \cite[Theorem 5.2 and Lemma 6.1]{Xi-Zhu-2017} 	$\process{X}$ will be strong Feller and open-set irreducible if $\inf_{x\in\R}|\upsigma(x)|>0$. Furthermore, under the same assumption, (\textbf{A2}) together with \cite[Theorem 6.1]{Meyn-Tweedie-AdvAP-III-1993} and \cite[Theorems 3.2 and 5.1]{Tweedie-1994} (by taking $\V(x)=x^2$) implies that $\process{X}$ admits a unique invariant measure  $\uppi$ such that for any $\upkappa>0$, \ben\lim_{t\nearrow\infty}\e^{\upkappa t}\lVert p^t(x,\D y)-\uppi(\D y)\rVert_{\mathcal{TV}}=0,\qquad x\in\R.\een
 However, in many interesting situations the diffusion coefficient $\upsigma$ can be singular, i.e. it can vanish (see e.g. \cite{Masoliver-Garrido-Llosa-1987} and \cite{Drozdov-Talkner-1998}). 

\subsection{Equilibria of the overdamped Langevin equation \eqref{eq:langevin} with singular noise term}

Assume (\textbf{A1}), (\textbf{A2}) and

\medskip
\begin{description}
	\item [A3] There is a bounded open interval $I\subset\R$ such that
	
	\medskip
	
	\begin{itemize}
		\item [(i)] $\inf_{x\in I}|\upsigma(x)|>0;$
		
		\medskip
		
		\item[(ii)] $\sup_{x\in K}\Exp^x(\uptau_I)<\infty$ for any compact $K\subset\R,$ where $\uptau_I:=\inf\{t\ge0:X_t\in I\}.$
	\end{itemize}
\end{description}

\medskip

\noindent Then, according to \cite[Theorems 4.1 and 4.2, and Corollary 4.4]{Khasminskii-Book-2012} $\process{X}$ admits a unique invariant measure $\uppi$ such that 
\be\label{eq:x} \lim_{t\nearrow\infty}\frac{1}{t}\int_0^t f(X_s)\D s=\int_{\R}f(y)\uppi(\D y),\qquad \Prob^x\text{-a.s.},\ee for every $x\in\R$ and $f\in\mathcal{B}_b(\R).$

\begin{proposition}\label{p2.2} The process $\process{X}$ will satisfy (\textbf{A3}) if there is an open interval $I$, containing $-1$, $0$ and $1$, such that $\inf_{x\in I}|\upsigma(x)|>0.$ 
	\end{proposition}
\begin{proof} 
	By assumption, there is $0<\upvarepsilon<1$ such that $I_\upvarepsilon:=(-1-\upvarepsilon, 1+\upvarepsilon)\subset I$. Thus, in particular $\inf_{x\in I_\upvarepsilon}|\upsigma(x)|>0.$
	Let us now show that $\sup_{x\in K}\Exp^x(\uptau_{I_\upvarepsilon})<\infty$ for any compact $K\subset\R$. Clearly, it suffices to prove the assertion for compact subsets of $I_{\upvarepsilon}^c$ only. Let   $\V:\R\to\R_+$, $\V\in \mathcal{C}^2(\R),$ be such that $\V(x)=|x|\I_{_{I_\upvarepsilon ^c}}(x)$. Further, for $n\in\N$ define $\uptau_n:=\inf\{t\ge0:|X_t|\ge n\}$. Clearly, since $\process{X}$ is conservative, $\uptau_{n}\nearrow\infty$, as $n\nearrow\infty$. Now, due to the martingale property of the process $\{M^\V_{t\wedge\uptau_{I_\upvarepsilon}\wedge\uptau_n}\}_{t\ge0}$ (defined in \eqref{eq:martingale}), we have
	\ben \Exp^x(\V(X_{t\wedge\uptau_{I_\upvarepsilon}\wedge\uptau_n}))-\V(x)=\Exp^x\left(\int_{0}^{t\wedge\uptau_{I_\upvarepsilon}\wedge\uptau_n}\mathcal{L}\V(X_s)\D s\right),\qquad t\ge0,\ x\in\R,\ n\in\N.\een 
	In particular, for $x\in I_{\upvarepsilon}^c$, $x<-1$, we have \ban \V(x)&\ge-\Exp^x\left(\int_{0}^{t\wedge\uptau_{I_\upvarepsilon}\wedge\uptau_n}\mathcal{L}\V(X_s)\D s\right)\\&=
	\Exp^x\left(\int_{0}^{t\wedge\uptau_{I_\upvarepsilon}\wedge\uptau_n}(-X_s^3+X_s)\D s\right)\\&\ge\left((1+\upvarepsilon)^3-(1+\upvarepsilon)\right)\Exp^x(t\wedge\uptau_{I_\upvarepsilon}\wedge\uptau_n),\qquad t\ge0,\ n\in\N.\ean By letting $t\nearrow\infty$ and $n\nearrow\infty$ we conclude \ben \Exp^x(\uptau_{I_\upvarepsilon})\le\frac{|x|}{(1+\upvarepsilon)^3-(1+\upvarepsilon)},\qquad x\in _{I_\upvarepsilon}^c,\ x<-1.\een Analogously, for $x\in _{I_\upvarepsilon}^c$, $x>1$, we  have \ben\Exp^x(\uptau_{I_\upvarepsilon})\le\frac{|x|}{(1+\upvarepsilon)^3-(1+\upvarepsilon)},\een which concludes the proof. 
	\end{proof} 
	
	\begin{remark} Let us remark that the above results can be slightly generalized. Namely, according to \cite[Theorem 5.2]{Abundo-2000} and \cite[Lemma 4.6]{Khasminskii-Book-2012}, $\process{X}$ will satisfy (\textbf{A3}) if there is $0<\upvarepsilon<1/2$  such that
		
		\medskip
		
		 \begin{itemize}
			\item [(i)] $\upsigma$ does not vanish on $I_\upvarepsilon(-1)\cup I_\upvarepsilon(0)\cup I_\upvarepsilon(1);$
			
			\medskip
			
			\item[(ii)] $\sup_{x\in [-1,0]}\Exp^x(\uptau_{I_\upvarepsilon(-1)}\vee\uptau_{I_\upvarepsilon(0)})<\infty$ and $\sup_{x\in [0,1]}\Exp^x(\uptau_{I_\upvarepsilon(0)}\vee\uptau_{I_\upvarepsilon(1)})<\infty$.
		\end{itemize} 
		\end{remark}

	\begin{proof}[Proof of Theorem \ref{tmx}] The first assertion follows from Proposition \ref{p2.2} and \cite[Corollary 4.4]{Khasminskii-Book-2012}.  To prove the second assertion we proceed as follows. From \eqref{eq:x} we automatically conclude that $\process{X}$ is \textit{$\uppi$-irreducible}, i.e.  \ben\int_0^\infty p^t(x,B)\D t>0,\qquad x\in\R,\een whenever $\uppi(B)>0$, $B\in \mathfrak{B}(\R)$. Here, $\mathfrak{B}(\R)$ denotes the Borel $\sigma$-algebra on $\R$.
		Next, \cite[Lemma 4.8]{Khasminskii-Book-2012} implies that the support of $\uppi$ has a non-empty interior, which together with the fact that $\process{X}$ is a $C_b$-Feller process, \eqref{eq:x} and \cite[Theorems 3.4 and 7.1]{Tweedie-1994} implies that $\process{X}$ is \textit{positive Harris recurrent process}, i.e. there is a $\sigma$-finite measure $\upvarphi$ such that \ben \Prob^x\left(\int_0^\infty \I_B(X_t)\D t=\infty\right)=1,\qquad x\in\R,\een whenever $\upvarphi(B)>0$, $B\in\mathfrak{B}(\R)$. Now, according to \cite[Theorems 5.1 and 7.1]{Tweedie-1994} and  \cite[Theorem 6.1]{Meyn-Tweedie-AdvAP-III-1993} (by taking $\V(x)=x^2$), the assertion will follow if we show that there is a $\sigma$-finite measure $\upphi$, whose support has a non-empty interior,  such that \be\label{eq:xx} \sum_{n=1}^\infty p^n(x,B)>0,\qquad x\in\R,\ee whenever $\upphi(B)>0,$  $B\in\mathfrak{B}(\R)$.
		Due to \cite[Theorems 7.3.6 and 7.3.7]{Durrett-Book-1996} there is a function  $p^t(x,y)>0$, $t>0,$ $x,y\in \bar{I}$, jointly continuous in $t$ and $x,y$, and $\mathcal{C}^2$ in $x$ on $I$, satisfying \ben \Exp^x(f(X_t),\uptau_{\bar{I}^c}>t)=\int_{I}p^t(x,y)f(y)\D y,\qquad t>0,\ x\in I,\ f\in\mathcal{C}_b(\R),\een where $\uptau_{\bar{I}^c}:=\inf\{t\ge0:X_t\in \bar{I}^c\}.$ Clearly, by employing dominated convergence theorem, the above relation holds also for any open interval $J\subseteq I$. Denote by $\mathcal{D}$ the class of all  $B\in\mathfrak{B}(I)$ (the Borel $\sigma$-algebra on $I$) such that \ben \Prob^x(X_t\in B,\ \uptau_{\bar{I}^c}>t)=\int_{B}p^t(x,y)\D y,\qquad t>0,\ x\in I.\een Clearly, $\mathcal{D}$ contains the $\pi$-system of open intervals in $I$, and forms a $\lambda$-system. Hence, by employing the famous Dynkin's $\pi$-$\lambda$ theorem we conclude that $\mathcal{D}=\mathfrak{B}(I).$ Consequently, for any $t>0$, $x\in I$ and $B\in\mathfrak{B}(\R)$ we have that
		\ben p^t(x,B)\ge \int_{B\cap I}p^t(x,y)\D y.\een Let us now take $\upphi(\cdot):=\uplambda(\cdot\cap I)$ and prove \eqref{eq:xx}, where $\uplambda$ stands for the Lebesgue measure on $\R$. Let  $x\in I^c$ (for $x\in I$ the assertion is obvious) and $B\in\mathfrak{B}(\R)$, $\upphi(B)>0$, be arbitrary. Then,
		\ben 
		\sum_{n=1}^\infty p^n(x,B)\ge\int_I\sum_{n=1}^\infty p^{n-t}(x,\D y)p^{t}(y,B),\qquad 0<t<1.
		\een Since $p^{t}(y,B)>0$ for $y\in I$, it suffices to show that \ben\sum_{n=1}^\infty p^{n-t}(x,I)\ge\Prob^x\left(\bigcup_{n=1}^\infty\{X_{n-t}\in I\}\right)>0\een for some $0<t<1$. Assume this is not the case, i.e. that \ben \Prob^x\left(\bigcup_{n=1}^\infty\{X_{n-t}\in I\}\right)=0,\qquad 0<t<1.\een This, in particular, implies that \ben\Prob^x\left(\bigcup_{q\in\Q\setminus\ZZ}\{X_{q}\in I\}\right)=0,\een which is impossible since $\process{X}$ has continuous sample paths, $I$ is an open set and, by assumption, $\Prob^x(\uptau_I<\infty)=1$ for every $x\in\R$. Thus, \ben 
		\sum_{n=1}^\infty p^n(x,B)>0,\qquad x\in\R,\een whenever $\upphi(B)>0$, which concludes the proof.
	\end{proof}

	The crucial assumption in the above discussion was that $\upsigma$ does not vanish at the roots of $\nabla V$, i.e. at $-1$, $0$ and $1$. Recall, 	if $\upsigma$ vanishes at $x_e\in\{-1,0,1\}$, then $\Prob^{x_e}(X_t=x_e,\ t\ge0)=1$. In particular, $\updelta_{x_e}$ is an invariant measure for $\process{X}$.
	
	\begin{proposition}
		If $\upsigma$ vanishes at $x_e\in\{-1,0,1\}$, then $\Prob^{x}(X_t>x_e,\ t\ge0)=1$ for all $x>x_e$, and $\Prob^{x}(X_t<x_e,\ t\ge0)=1$ for all $x<x_e$. 
		\end{proposition}
\begin{proof}
	Let us discuss the case when $x_e=0$ and $x>0$. The oder five cases are treated in a similar way, simply by appropriately shifting and/or mirroring the function $\V$ defined below. 
	We follow the proof of  \cite[Lemma 1]{Liu-Wang-2011}.
	For $n\in\N$ define $\uptau_n:=\inf\{t\ge0:X_t\notin(1/n,n)\}$. Clearly, $\{\tau_n\}_{n\in\N}$ is a non-decreasing sequence of stopping times. Set $\uptau_\infty:=\lim_{n\nearrow\infty}\uptau_n.$ In the sequel we show that $\Prob^x(\tau_\infty=\infty)=1$ for all $x>0$, which automatically implies the assertion. Assume  this is not the case. Then there  exist $x_0>0$ and $0<\upvarepsilon<1$, such that   $\Prob^{x_0}(\uptau_\infty<\infty)>\upvarepsilon$. This automatically implies  that there are $n_0\in\N$ and $T>0$, such that $x_0\in(1/n,n)$ and $\Prob^{x_0}(\uptau_n<T)>\upvarepsilon$ for all $n\ge n_0$. Next, define $\V(x):=x-1-\ln x$. It is elementary to check that $\V:(0,\infty)\to\R_+$ and $\V\in \mathcal{C}^2(0,\infty)$. Also, for $n\in\N$ let $\V_n\in\mathcal{C}^2(\R)$ be such that $\V_n(x)\I_{(1/n,n)}(x)=\V(x).$ Now, by the  martingale property of $\{M^{\V_n}_{t\wedge\uptau_n}\}_{t\ge0}$ (defined in \eqref{eq:martingale}), we have that for all $n\in\N$,
	\ban
	&\Exp^{x_0}(\V_n(X_{T\wedge\uptau_n}))-\V_n(x_0)\\&= \Exp^{x_0}(\V(X_{T\wedge\uptau_n}))-\V(x_0)\\&=\Exp^{x_0}\left(\int_{0}^{T\wedge\uptau_n}\mathcal{L}\V(X_s)\D s\right)\\
	&=\Exp^{x_0}\left(\int_{0}^{T\wedge\uptau_n}\left(\left(-X_s^3+X_s\right)\left(1-X_s^{-1}\right)+\frac{1}{2}\upsigma^2(X_s)X_s^{-2}\right)\D s\right)\\
	&=\Exp^{x_0}\left(\int_{0}^{T\wedge\uptau_n}\left(-X_s^{3}+X^{2}_s+X_s-1+\frac{1}{2}\upsigma^2(X_s)X_s^{-2}\right)\D s\right).\ean
	According to (\textbf{A2}) there is $r_1>0$ such that $|\upsigma(x)|\le \sqrt{2}x^2$ for all $x\ge r_1.$ Next, take $r_2\ge r_1$ such that $-x^3+2x^2+x-1\le0$ for all $x\ge r_2$. Thus, \ben\sup_{x\ge r_2}\left(-x^{3}+x^{2}+x-1+\frac{1}{2}\upsigma^2(x)x^{-2}\right)\le 0.\een
	On the other hand, in order to bound the above term on $[0,r_2]$, we first observe that due to (\textbf{A1}), compactness of  $[0,r_2]$ and $\upsigma(0)=0$ there is $c>0$ such that $|\upsigma(x)|\le\sqrt{2c}|x|$ for all $x\in[0,r_2]$. Thus, 
	\ben\sup_{0\le x\le r_2}\left(-x^{3}+x^{2}+x-1+\frac{1}{2}\upsigma^2(x)x^{-2}\right)\le\sup_{0\le x\le r_2}\left(-x^{3}+x^{2}+x-1+c\right)\le c .\een
	We conclude now that \ben\Exp^{x_0}(\V_n(X_{T\wedge\uptau_n}))-\V_n(x_0)\le c\Exp^{x_0}(T\wedge\uptau_n)\le c T.\een Thus, for $n\ge n_0$ we have that \ben (\V(1/n)\wedge\V(n))\Prob^{x_0}(\uptau_n<T)\le\Exp^{x_0}(\V(X_{\uptau_n})\I_{\{\uptau_n< T\}})\le\Exp^{x_0}(\V_n(X_{T\wedge\uptau_n}))\le \V(x_0)+c T.\een Now, since $\Prob^{x_0}(\uptau_n<T)>\upvarepsilon$ for all $n\ge n_0$ and $\V(1/n)\wedge\V(n)\nearrow\infty$, as $n\nearrow\infty$, the assertion follows. 
\end{proof}

At the end,  we assume  $\upsigma$ has a root at $x_e\in\{-1,0,1\}$ and discuss stability of $\process{X}.$ 
We start with an auxiliary result. For $a\in\R$ and $0<r_1<r_2$ define $I_{r_1,r_2}(a):=\{x\in\R:r_1<|x-a|<r_2\}$.
\begin{lemma}\label{lm2.5} 
\begin{itemize}
	\item [(i)] If $x_e=0$, then for any $0<\upvarepsilon_1<\upvarepsilon_2<1$ we have that \ben\sup_{x\in I_{\upvarepsilon_1,\upvarepsilon_2}(x_e)}\Exp^x(\uptau_{I^c_{\upvarepsilon_1,\upvarepsilon_2}(x_e)})\le\frac{\upvarepsilon_2-\upvarepsilon_1}{(\upvarepsilon_1-\upvarepsilon_1^3)\wedge(\upvarepsilon_2-\upvarepsilon_2^3)}.\een
	
	\medskip
	
	\item[(ii)] If $x_e\in\{-1,1\}$, then for any $0<\upvarepsilon_1<\upvarepsilon_2<1$ we have that \ben\sup_{x\in I_{\upvarepsilon_1,\upvarepsilon_2}(x_e)}\Exp^x(\uptau_{I^c_{\upvarepsilon_1,\upvarepsilon_2}(x_e)})\le\frac{\upvarepsilon_2}{(\upvarepsilon_1-\upvarepsilon_1^3)\wedge(\upvarepsilon_2-\upvarepsilon_2^3)}.\een
\end{itemize}	
	\end{lemma} 
	\begin{proof}
		 Let   $\V:\R\to\R_+$, $\V\in \mathcal{C}^2(\R),$ be such that $\V(x)=|x-x_e|\I_{I^c_{\upvarepsilon_1}(x_e)}(x)$.  Again, due to the martingale property of  $\{M^\V_{t\wedge\uptau_{I^c_{\upvarepsilon_1,\upvarepsilon_2}(x_e)}}\}_{t\ge0}$, we have
		\ben \Exp^x(\V(X_{t\wedge\uptau_{I^c_{\upvarepsilon_1,\upvarepsilon_2}(x_e)}}))-\V(x)=\Exp^x\left(\int_{0}^{t\wedge\uptau_{I^c_{\upvarepsilon_1,\upvarepsilon_2}(x_e)}}\mathcal{L}\V(X_s)\D s\right),\qquad t\ge0,\ x\in\R.\een 
	\begin{itemize}
		\item [(i)] If $x_e=0$, then for $x\in I_{\upvarepsilon_1,\upvarepsilon_2}(x_e)$ we have \ban \upvarepsilon_2-\upvarepsilon_1&\ge \Exp^x(\V(X_{t\wedge\uptau_{I^c_{\upvarepsilon_1,\upvarepsilon_2}(x_e)}}))-\V(x)
		\\&=
		\Exp^x\left(\int_{0}^{t\wedge\uptau_{I^c_{\upvarepsilon_1,\upvarepsilon_2}(x_e)}}{\rm sgn}\,(X_s)(-X_s^3+X_s)\D s\right)\\&\ge((\upvarepsilon_1-\upvarepsilon_1^3)\wedge(\upvarepsilon_2-\upvarepsilon_2^3))\Exp^x(t\wedge\uptau_{I^c_{\upvarepsilon_1,\upvarepsilon_2}(x_e)}),\qquad t\ge0,\ean where $sgn$ denotes the signum function. Finally, by letting $t\nearrow\infty$ the assertion follows. 
		
		\medskip
		
		\item [(ii)] 
		If $x_e\in\{-1,1\}$, then for $x\in I_{\upvarepsilon_1,\upvarepsilon_2}(x_e)$ we have \ban \upvarepsilon_2&\ge\V(x)\\
		&\ge
		-\Exp^x\left(\int_{0}^{t\wedge\uptau_{I^c_{\upvarepsilon_1,\upvarepsilon_2}(x_e)}}{\rm sgn}\,(X_s-x_e)(-X_s^3+X_s)\D s\right)\\&\ge((\upvarepsilon_1-\upvarepsilon_1^3)\wedge(\upvarepsilon_2-\upvarepsilon_2^3))\Exp^x(t\wedge\uptau_{I^c_{\upvarepsilon_1,\upvarepsilon_2}(x_e)}),\ean 
		where in the last line we used the fact that \ben(1+r)^3-(1+r)>r-r^3,\qquad 0<r<1.\een
Finally, by letting $t\nearrow\infty$ the desired result follows.	
	\end{itemize}

		\end{proof}
		
		Now, we are ready to prove  Theorem \ref{tm1.1}.

	\begin{proof}[Proof of Theorem \ref{tm1.1}]
		\begin{itemize}
			\item [(i)] According to  Lemma \ref{lm2.5} and \cite[Theorem 5.5]{Khasminskii-Book-2012} it suffices to show that there are $\upvarepsilon>0$ and non-negative $\V\in\mathcal{C}^2(\R\setminus\{0\})$, such that $\lim_{|x|\searrow0}\V(x)=\infty$ and $\mathcal{L}\V(x)\le0$ for $0<|x|<\upvarepsilon$.	
			By assumption, there is $0<\upvarepsilon_0<1$  such that \ben\upkappa_{\upvarepsilon}:=\inf\left\{\upkappa:\frac{|\upsigma(x)|}{|x|}\le\upkappa,\ 0<|x|<\upvarepsilon\right\}<\sqrt{2}\een for every $0<\upvarepsilon\le\upvarepsilon_0.$ Observe that this is well defined due to (\textbf{A1}). Thus, there is $0<\upvarepsilon_1\le\upvarepsilon_0$ such that for any $0<\upvarepsilon\le\upvarepsilon_1$,
			$\upkappa_{\upvarepsilon}^2+2\upvarepsilon^2<\upkappa^2_{\upvarepsilon}+2-\upkappa^2_{\upvarepsilon_0}\le2.$ In particular, $\upkappa^2_{\upvarepsilon}<2(1-\upvarepsilon^2)$ for $0<\upvarepsilon\le\upvarepsilon_1.$ Now, fix $0<\upvarepsilon\le\upvarepsilon_1$, and $\upalpha>0$ such that \ben\frac{1}{\upalpha+1}\ge\frac{\upkappa^2_{\upvarepsilon}}{2(1-\upvarepsilon^2)}.\een
			Further, define $\V(x):=|x|^{-\alpha}.$ For $0<|x|<\upvarepsilon$ we have that 
			\ban \mathcal{L}\V(x)&=-\upalpha(-x^3+x)|x|^{-\upalpha-1}{\rm sgn}\,(x)+\frac{\upalpha(\upalpha+1)}{2}\upsigma^2(x)|x|^{-\upalpha-2}\\
			&\le-\upalpha(-x^3+x)|x|^{-\upalpha-1}{\rm sgn}\,(x)+\frac{\upalpha(\upalpha+1)}{2}\upkappa^2_{\upvarepsilon}|x|^{-\upalpha}.\ean If $\upkappa_{\upvarepsilon}=0$, then $\mathcal{L}\V(x)\le0$ for $0<|x|<\upvarepsilon$, and if $\upkappa_{\upvarepsilon}>0$, then
			\ben\mathcal{L}\V(x)\le\upalpha|x|^{-\upalpha+2}-\upalpha|x|^{-\upalpha}+\upalpha(1-\upvarepsilon^2)|x|^{-\upalpha}=\upalpha|x|^{-\upalpha+2}-\upalpha\upvarepsilon^2|x|^{-\upalpha}<0\een for $0<|x|<\upvarepsilon$, which proves the assertion.
			
			\medskip
			
			\item[(ii)]  We use the same strategy as in (i). Let $\V(x):=\ln\ln\left(1/|x|+\e\right).$ For $0<|x|<\updelta$, we have that \ban\mathcal{L}\V(x)&=\frac{-(-x^3+x){\rm sgn }\,x}{(|x|+\e|x|^2)\ln\left(1/|x|+\e\right)}+\frac{-1+\ln(1/|x|+\e)(1+2\e |x|)}{(1+\e |x|)^2\ln^2(1/|x|+\e)}\\
			&=\frac{(1+\e|x|)(|x|^2-1)\ln(1/|x|+\e)-1+\ln(1/|x|+\e)(1+2\e |x|)}{(1+\e|x|)^2\ln^2\left(1/|x|+\e\right)}\\
			&=\frac{|x|^2(1+\e|x|)\ln(1/|x|+\e)+\e|x|\ln(1/|x|+\e)-1}{(1+\e|x|)^2\ln^2\left(1/|x|+\e\right)}.
			\ean Now, by observing that \ben\lim_{|x|\searrow0}|x|\ln\left(\frac{1}{|x|}+\e\right)=0,\een we conclude that there is $0<\upvarepsilon<\updelta$ such that $\mathcal{L}\V(x)\le0$ for $0<|x|<\upvarepsilon.$
			
			\medskip
			
			\item[(iii)] According to  Lemma \ref{lm2.5} and \cite[Theorem 5.6]{Khasminskii-Book-2012} it suffices to show that there are $\upvarepsilon>0$ and non-negative $\V\in\mathcal{C}^2(\R\setminus\{0\})$, such that $\V(x)>0$ for $x\in \R\setminus\{0\}$, $\lim_{|x|\searrow0}\V(x)=0$ and  $\mathcal{L}\V(x)\le0$ for $0<|x|<\upvarepsilon$.  By assumption, there is $\upvarepsilon_0>0$  such that \ben\upkappa_{\upvarepsilon}:=\sup\left\{\upkappa:\frac{|\upsigma(x)|}{|x|}\ge\upkappa,\ 0<|x|<\upvarepsilon\right\}>\sqrt{2}\een for every $0<\upvarepsilon\le\upvarepsilon_0.$  Now, fix $0<\upvarepsilon\le\upvarepsilon_0$, and $\upalpha>0$ such that \ben0<\upalpha\le 1-\frac{2}{\upkappa^2_\upvarepsilon}.\een
			Further, define $\V(x):=|x|^{\alpha}.$ For $0<|x|<\upvarepsilon$ we have that 
			\ban \mathcal{L}\V(x)&=\upalpha(-x^3+x)|x|^{\upalpha-1}{\rm sgn}\,(x)+\frac{\upalpha(\upalpha-1)}{2}\upsigma^2(x)|x|^{\upalpha-2}\\
			&\le-\upalpha|x|^{\upalpha+2}+\upalpha|x|^{\upalpha}+\frac{\upalpha(\upalpha-1)}{2}\upkappa^2_{\upvarepsilon}|x|^{\upalpha}\\&=-\upalpha|x|^{-\upalpha+2},\ean  which proves the desired result.

			\medskip
			
			\item[(iv)] Define $\V(x):=|x-x_e|$. Then, $\V\in\mathcal{C}^2(\R\setminus\{x_e\})$, $\V(x)>0$ for $|x-x_e|>0$, $\lim_{|x-x_e|\searrow0}\V(x)=0$ and  \ben\mathcal{L}\V(x)=(-x^3+x){\rm sgn}\,(x-x_e)\le0,\qquad 0<|x-x_e|<1.\een Thus, the first assertion follows from Lemma \ref{lm2.5} and \cite[Theorem 5.6]{Khasminskii-Book-2012}. 
			
			To prove the second assertion, according to \cite[Theorems 5.11 and 5.15]{Khasminskii-Book-2012}, it suffices to show that for $\V(x):=|x-x_e|^\upalpha$, $\upalpha>0$,  we have \ben \mathcal{L}\V(x)\le-c\V(x),\qquad x\in\R,\ xx_e>1.\een We have that
			\ban\mathcal{L}\V(x)&=\upalpha(-x^3+x)|x-x_e|^{\upalpha-1}{\rm sgn}(x-x_e)+\upalpha(\upalpha-1)\frac{\upsigma(x)^2}{2}|x-x_e|^{\upalpha-2}\\
			&=-\upalpha|x-x_e|^\upalpha\left(x(x+x_e)-(\upalpha-1)\frac{\upsigma(x)^2}{2}|x-x_e|^{-2}\right)\\
			&\le -c|x-x_e|^\upalpha\\
			&=-c\V(x),\qquad x\in\R,\ xx_e>1.\ean
				\end{itemize}
		\end{proof}

		\begin{remark}
			\begin{itemize}
				\item [(i)] If $0<\upalpha\le1$, then the constant $c$ in Theorem \ref{tm1.1} (iv) satisfies \ben 2+(1-\upalpha)\frac{\upkappa^2}{2}\ge \frac{c}{\upalpha}\ge 2,\een where \ben\upkappa:=\inf_{\upvarepsilon>0}\inf\left\{\upvarrho:\frac{|\upsigma(x)|}{|x-x_e|}\le\upvarrho,\ 0<|x-x_e|<\upvarepsilon\right\}.\een
				
					\medskip
				
				\item[(ii)] Assume $1<\upalpha\le2$, and let 
				\ben r:=\inf\{\uprho\ge2:\upsigma(x)^{2}\le2|x-x_e|^2x^2 \ \text{for}\ xx_e\ge\uprho \},\een which is finite according to (\textbf{A2}). Now, we have 
				\ban &\inf_{xx_e\ge r}\left(x(x+x_e)-(\upalpha-1)\frac{\upsigma(x)^2}{2}|x-x_e|^{-2}\right)\\&\ge\inf_{xx_e\ge r}\left((2-\upalpha)x^2+xx_e\right)\\&=(2-\upalpha)r^2+r. \ean
				Further, let \ben\upbeta:=\inf\{\upgamma\ge0:|\upsigma(x)|\le\upgamma|x-x_e|\ \text{for}\ 1\le xx_e\le r\},\een which is finite due to (\textbf{A1}), $\upsigma(x_e)=0$ and compactness of $[1,r]$, and assume that $\upbeta<2/\sqrt{\upalpha-1}$. Hence, \ban &\inf_{1\le xx_e\le r}\left(x(x+x_e)-(\upalpha-1)\frac{\upsigma(x)^2}{2}|x-x_e|^{-2}\right)\\&\ge\inf_{1\le xx_e\le r}\left(x(x+x_e)-(\upalpha-1)\frac{\upbeta^2}{2}\right)\\&=2 -(\upalpha-1)\frac{\upbeta^2}{2}.\ean Thus, \ben 2\ge \frac{c}{\upalpha}\ge2 -(\upalpha-1)\frac{\upbeta^2}{2}.\een
				
				\medskip
				
				\item[(iii)] Assume that $\upalpha>2$ and \ben\limsup_{|x|\nearrow\infty}|\upsigma(x)|/|x|^2<\sqrt{2}/\sqrt{\upalpha-1}.\een Let \ben r:=\inf\{\uprho\ge2:|\upsigma(x)|\le(\sqrt{2}/\sqrt{\upalpha-1})|x-x_e|x \ \text{for}\ xx_e\ge\uprho \},\een which is finite by assumption. Thus,
				\ben \inf_{xx_e\ge r}\left(x(x+x_e)-(\upalpha-1)\frac{\upsigma(x)^2}{2}|x-x_e|^{-2}\right)\ge r. \een Further, let $\upbeta$ be as in (ii). Then again 
			\ben \inf_{1\le xx_e\le r}\left(x(x+x_e)-(\upalpha-1)\frac{\upsigma(x)^2}{2}|x-x_e|^{-2}\right)\ge2 -(\upalpha-1)\frac{\upbeta^2}{2},\een which implies \ben 2\ge \frac{c}{\upalpha}\ge 2 -(\upalpha-1)\frac{\upbeta^2}{2}.\een
			\end{itemize}
			\end{remark}

\section{Stability of general diffusion processes}\label{s3}
We start with  proofs of Theorem \ref{tm3.1} and Corollary \ref{c1.1}.

	 	\begin{proof}[Proof of Theorem \ref{tm3.1}]  Observe first that $\Prob^x(X_t\neq x_e,\ t\ge0)=1$ for all $|x-x_e|>0$ (see \cite[Lemma 5.3]{Khasminskii-Book-2012}). 
	 			Next, for $n\in\N$ define $\uptau_n:=\inf\{t\ge0:|X_t|\ge n\}.$ 
	 			Clearly, under the  assumptions of the theorem, the process $\process{M^f}$ (defined in \eqref{eq:martingale}) is a local martingale for every $f\in\mathcal{C}^2(\R^d)$.
	 			Now, define  $f:[0,\infty)\times\R^d\setminus\{x_e\}\to[0,\infty)$ by \ben f(t,x):=\Phi_c^{-1}(t+\Phi_c\circ\V(x)).\een Clearly, $f$ is continuously differentiable with respect to the first variable on $(0,\infty)$, and twice continuously differentiable with respect to the second variable on $\R^d\setminus\{x_e\}.$ Next, note that the process \be\left\{\V(X_{t\wedge\uptau_n})+c\int_0^{t\wedge\uptau_n}\upvarphi\circ\V(X_s)\D s\right\}_{t\ge0}\ee is a supermartingale for any $n\in\N$. Indeed, for $t\ge s\ge0$, $|x-x_e|>0$ (for $x=x_e$ the assertion is obvious) and $n\in\N$, we have that
	 		\ban &\Exp^x\left(\V(X_{t\wedge\uptau_n})+c\int_0^{t\wedge\uptau_n}\upvarphi\circ\V(X_u)\D u\Big|\mathcal{F}_s\right)\\&=
	 		\Exp^x\left(\V(X_{t\wedge\uptau_n})+c\int_0^{t\wedge\uptau_n}\left(\upvarphi\circ\V(X_u)+\frac{\mathcal{L}\V(X_u)}{c}-\frac{\mathcal{L}\V(X_u)}{c}\right)\D u\Big|\mathcal{F}_s\right)\\&=
	 		\V(X_{s\wedge\uptau_n})-\int_0^{s\wedge\uptau_n}\mathcal{L}\V(X_u)\D u	+ c\int_0^{s\wedge\uptau_n}\left(\upvarphi\circ\V(X_u)+\frac{\mathcal{L}\V(X_u)}{c}\right)\D u	\\
	 		&\ \ \ +c\Exp^x\left(\int_{s\wedge\uptau_n}^{t\wedge\uptau_n}\left(\upvarphi\circ\V(X_u)+\frac{\mathcal{L}\V(X_u)}{c}\right)\D u	\Big|\mathcal{F}_s\right)\\
	 		&\le \V(X_{s\wedge\uptau_n})+c\int_0^{s\wedge\uptau_n}\upvarphi\circ\V(X_u)\D u.
	 		\ean
	 		Now, by using this fact, \cite[Corollary 4.5]{Hairer-Lecture-notes-2016} states that the process $\left\{f(s+t\wedge\uptau_n,X_{t\wedge\uptau_n}) \right\}_{t\ge0}$ is also supermartingale for any $s\ge0$ and $n\in\N$. 
	 		 In particular \ben \Exp^x\left( f(t\wedge\uptau_n,X_{t\wedge\uptau_n})\right)\le f(0,x)=\V(x),\qquad t\ge0,\ |x-x_e|>0.\een Consequently, by employing Fatou's lemma and conservativeness of $\process{X}$, $\{f(t,X_t)\}_{t\ge0}$ is also is a supermartingale. Furthermore, since it is positive, it converges $\Prob^x$-a.s. for all $|x-x_e|>0$. Consequently,  
	 		\ben \sup_{t\ge 0}f(t,X_t)\le Y_x\qquad \Prob^x\text{-a.s.,}\een
	 	where $Y_x$ is a strictly positive $\Prob^x$-finite random variable. Finally, we conclude that
	 	\ben \V(X_t)=|X_t-x_e|^\upalpha\le \Phi_c^{-1}(\Phi_c(Y_x)-t),\qquad\Prob^x\text{-a.s.},\ x\in\R^d,\ t\ge0.\een	 		
	 		\end{proof}
	 		
\begin{proof}[Proof of Corollary \ref{c1.1}]
	\begin{itemize}
		\item [(i)] 	Fix $0<\upvarepsilon<r$ and take concave and continuously differentiable function $\upvarphi$ such that 	\ben\upvarphi(t)=\left\{\begin{array}{cc}
			t, & 0\le t\le r-\upvarepsilon \\
			r,&
			t\ge r+\upvarepsilon 
		\end{array}\right.\le \left\{\begin{array}{cc}
		t, & 0\le t\le r \\
		r,&
		t\ge r. 
	\end{array}\right.\een The assertion now follows from Theorem \ref{tm3.1} and by observing that \ben\lim_{t\nearrow\infty}\frac{\ln\Phi^{-1}_c(\Phi_c(Y_x)-t)}{t}=-c,\qquad \Prob^x\text{-a.s.},\ x\in\R^d.\een

	\item[(ii)] Fix $0<\upvarepsilon<r_{\upbeta}$ and take concave and continuously differentiable function $\upvarphi$ such that $\upvarphi(t)\le\upvarphi_{\upbeta}(t)$ for $t>0$, and	$\upvarphi(t)=\upvarphi_{\upbeta}(t)$ for $t\in(r_\upbeta-\upvarepsilon,r_\upbeta+\upvarepsilon)^c.$ Again, the assertion  follows from Theorem \ref{tm3.1} and by observing that \ben\lim_{t\nearrow\infty}\frac{\ln\Phi^{-1}_c(\Phi_c(Y_x)-t)}{t^\upbeta}=-c^\upbeta,\qquad \Prob^x\text{-a.s.},\ x\in\R^d.\een
	\end{itemize}	 			 
\end{proof}	 
	
At the end we also conclude the following.

\begin{corollary}\label{c3.3}
	Assume the conditions of Theorem \ref{tm3.1}, 
and assume there are  $c>0$, $r>0$ and $0<\upgamma<1$, such that 	\ben\mathcal{L}\V(x)\le \left\{\begin{array}{cc}
	-c(\V(x))^\upgamma, &  0<\V(x)\le r \\
	-cr^\upgamma,&
	\V(x)\ge r.
\end{array}\right. \een  Then,
 \ben\limsup_{t\nearrow\infty}\frac{\ln |X_t-x_e|}{t^\upbeta}\le-\frac{c^\upbeta}{\upalpha}\een for any $\upbeta\ge1$ and $x\in\R^d.$ 
\end{corollary}	
\begin{proof}
	The assertion follows from the fact that  \ben t^\upgamma\ge t  \qquad\text{and}\qquad t^\upgamma\ge \upvarphi_\upbeta(t)\een for all $t>0$ small enough.
\end{proof}

\begin{proposition}\label{p3.4} In the one-dimensional case, the condition in Corollary \ref{c3.3}   will hold if there are $c$, $r>0$ and $0<\upgamma<1$, such that 
	
	\medskip
	 
	\begin{itemize}
		\item [(i)] 	${\rm sgn}(x-x_e)b(x)\le -c|x-x_e|^\upgamma$ for $|x-x_e|\le r$;
	
	\medskip
	
	\item[(ii)] $\sup_{|x-x_e|\ge r}{\rm sgn}(x-x_e)b(x)\le -c r^\upgamma.$
	\end{itemize}  
\end{proposition}		
\begin{proof}
Take $\V(x):=|x-x_e|.$ Then, for $|x-x_e|>0$, we have that
\ben\mathcal{L}V(x)={\rm sgn}(x-x_e)b(x)\le \left\{\begin{array}{cc}
	-c(\V(x))^{\upgamma}, &  0<\V(x)\le r \\
	-c r^\upgamma,&
	\V(x)\ge r.
\end{array}\right.\een
Observe that in Proposition \ref{p3.4} we deal with diffusion processes with H\"older continuous coefficients. For existence, uniqueness and structural properties of such processes see  \cite{Fang-Zhang-2005}, \cite{Lan-Wu-2014} and  \cite{Xi-Zhu-2017}. 

\end{proof}	 		
\section*{Acknowledgement}Financial support through the Croatian Science Foundation (under Project 8958)   is gratefully acknowledged.

\bibliographystyle{alpha}
\bibliography{References}

\def\cprime{$'$}
\begin{thebibliography}{ABW10}

\bibitem[Abu00]{Abundo-2000}
M.~Abundo.
\newblock On first-crossing times of one-dimensional diffusions over two
  time-dependent boundaries.
\newblock {\em Stochastic Anal. Appl.}, 18(2):179--200, 2000.

\bibitem[ABW10]{Albeverio-Brzezniak-Wu-2010}
S.~Albeverio, Z.~Brze\'zniak, and J-L. Wu.
\newblock Existence of global solutions and invariant measures for stochastic
  differential equations driven by {P}oisson type noise with non-{L}ipschitz
  coefficients.
\newblock {\em J. Math. Anal. Appl.}, 371(1):309--322, 2010.

\bibitem[Col79]{Coleman-Book-1979}
S.~Coleman.
\newblock The uses of instantons.
\newblock In A.~Zichichi, editor, {\em The Whys of Subnuclear Physics}, volume
  vol 15 of {\em The Subnuclear Series}. Springer, Boston, MA, 1979.

\bibitem[DT98]{Drozdov-Talkner-1998}
A.~N. Drozdov and P~Talkner.
\newblock Path integrals for {F}okker-{P}lanck dynamics with singular
  diffusion: {A}ccurate factorization for the time evolution operator.
\newblock {\em J. Chem. Phys.}, 109(6):2080--2091, 1998.

\bibitem[Dur96]{Durrett-Book-1996}
R.~Durrett.
\newblock {\em Stochastic calculus}.
\newblock CRC Press, Boca Raton, FL, 1996.

\bibitem[FZ05]{Fang-Zhang-2005}
S.~Fang and T.~Zhang.
\newblock A study of a class of stochastic differential equations with
  non-{L}ipschitzian coefficients.
\newblock {\em Probab. Theory Related Fields}, 132(3):356--390, 2005.

\bibitem[Hai16]{Hairer-Lecture-notes-2016}
M.~Hairer.
\newblock {\em Convergence of {M}arkov processes}.
\newblock Lecture notes, University of Warwick. Available at
  http://www.hairer.org/notes/Convergence.pdf, 2016.

\bibitem[Kha12]{Khasminskii-Book-2012}
R.~Khasminskii.
\newblock {\em Stochastic stability of differential equations}.
\newblock Springer, Heidelberg, second edition, 2012.

\bibitem[Leh70]{Lehn-Book-1970}
J.~M. Lehn.
\newblock Nitrogen inversion.
\newblock In {\em Dynamic Stereochemistry}, volume 15/3 of {\em Fortschritte
  der Chemischen Forschung}. Springer, Berlin, Heidelberg, 1970.

\bibitem[LMK92]{Liang-Muller-Kirsten-1992}
J-Q. Liang and H.~J.~W. M\"uller-Kirsten.
\newblock Periodic instantons and quantum-mechanical tunneling at high energy.
\newblock {\em Phys. Rev. D}, 46:4685--4690, 1992.

\bibitem[LW11]{Liu-Wang-2011}
M.~Liu and K.~Wang.
\newblock Persistence and extinction in stochastic non-autonomous logistic
  systems.
\newblock {\em J. Math. Anal. Appl.}, 375(2):443--457, 2011.

\bibitem[LW14]{Lan-Wu-2014}
G.~Lan and J.-L. Wu.
\newblock New sufficient conditions of existence, moment estimations and non
  confluence for {SDE}s with non-{L}ipschitzian coefficients.
\newblock {\em Stochastic Process. Appl.}, 124(12):4030--4049, 2014.

\bibitem[MGL87]{Masoliver-Garrido-Llosa-1987}
J.~Masoliver, L.~Garrido, and J.~Llosa.
\newblock Geometrical derivation of the intrinsic {F}okker-{P}lanck equation
  and its stationary distribution.
\newblock {\em J. Statist. Phys.}, 46(1-2):233--248, 1987.

\bibitem[MT93a]{Meyn-Tweedie-AdvAP-II-1993}
S.~P. Meyn and R.~L. Tweedie.
\newblock Stability of {M}arkovian processes. {II}. {C}ontinuous-time processes
  and sampled chains.
\newblock {\em Adv. in Appl. Probab.}, 25(3):487--517, 1993.

\bibitem[MT93b]{Meyn-Tweedie-AdvAP-III-1993}
S.~P. Meyn and R.~L. Tweedie.
\newblock Stability of {M}arkovian processes. {III}. {F}oster-{L}yapunov
  criteria for continuous-time processes.
\newblock {\em Adv. in Appl. Probab.}, 25(3):518--548, 1993.

\bibitem[Nel67]{Nelson-Book-1967}
E.~Nelson.
\newblock {\em Dynamical theories of {B}rownian motion}.
\newblock Princeton University Press, Princeton, N.J., 1967.

\bibitem[PR07]{Prevot-Rockner-Book-2007}
C.~Pr\'ev\^ot and M.~R\"ockner.
\newblock {\em A concise course on stochastic partial differential equations}.
\newblock Springer, Berlin, 2007.

\bibitem[Twe94]{Tweedie-1994}
R.~L. Tweedie.
\newblock Topological conditions enabling use of {H}arris methods in discrete
  and continuous time.
\newblock {\em Acta Appl. Math.}, 34(1-2):175--188, 1994.

\bibitem[XZ17]{Xi-Zhu-2017}
F.~Xi and C.~Zhu.
\newblock Jump type stochastic differential equations with non-{L}ipschitz
  coefficients: non confluence, {F}eller and strong {F}eller properties, and
  exponential ergodicity.
\newblock {\em Preprint. Available at https://arxiv.org/abs/1706.01393v3},
  2017.

\end{thebibliography}

\end{document}